\documentclass[11pt]{amsart}

\usepackage{amsmath}
\usepackage{amssymb}
\usepackage{amsfonts}
\usepackage{amsthm}
\usepackage{enumerate}
\usepackage{bbm}
\usepackage[mathscr]{eucal}
\usepackage{graphicx}
\usepackage{verbatim}
\usepackage{hyperref}

\newtheorem{theorem}{Theorem}

\newtheorem{corollary}{Corollary}
\newtheorem{proposition}{Proposition}

\theoremstyle{definition}

\theoremstyle{remark}

\def\Xint#1{\mathchoice
{\XXint\displaystyle\textstyle{#1}}%
{\XXint\textstyle\scriptstyle{#1}}%
{\XXint\scriptstyle\scriptscriptstyle{#1}}%
{\XXint\scriptscriptstyle\scriptscriptstyle{#1}}%
\!\int}
\def\XXint#1#2#3{{\setbox0=\hbox{$#1{#2#3}{\int}$ }
\vcenter{\hbox{$#2#3$ }}\kern-.6\wd0}}

\def\dashint{\Xint-}

\begin{document}

\title{A Fourier restriction theorem based on convolution powers}
\author{Xianghong Chen}
\address{X. Chen\\Department of Mathematics\\University of Wisconsin-Madison\\Madison, WI 53706, USA}
\curraddr{}
\email{xchen@math.wisc.edu}
\thanks{This research was supported in part by NSF grant 0652890.}

\dedicatory{}
\commby{}

\begin{abstract}
We prove a Fourier restriction estimate under the assumption that certain convolution power of the measure admits an $r$-integrable density.
\end{abstract}
\maketitle

\section*{Introduction}
Let $\mathcal F$ be the Fourier transform defined on the Schwartz space by
$$\hat f(\xi)=\int_{\mathbb R^d}e^{-2\pi i \langle\xi,x\rangle}f(x)dx$$
where $\langle\xi,x\rangle$ is the Euclidean inner product. We are interested in Borel measures $\mu$ defined on $\mathbb R^d$ for which $\mathcal F$ maps $L^p(\mathbb R^d)$ boundedly to $L^2(\mu)$; i.e.
\begin{equation}\label{restriction}
\|\hat f\|_{L^2(\mu)}\lesssim \|f\|_{L^p(\mathbb R^d)}, \forall f\in\mathcal S(\mathbb R^d).
\end{equation}
Here ``$\lesssim$'' means the left-hand side is bounded by the right-hand side multiplied by a positive constant that is independent of $f$.

If $\mu$ is a singular measure, then such result can be interpreted as a restriction property of the Fourier transform. Such restriction estimates for singular measures were first obtained by Stein in the 1960's. If $\mu$ is the surface measure on the sphere, the Stein-Tomas theorem \cite{tomas}, \cite{stein} states that \eqref{restriction} holds for $1\le p\le \frac{2(d+1)}{d+3}$. Mockenhaupt \cite{mockenhaupt} and Mitsis \cite{mitsis} have shown that Tomas's argument in \cite{tomas} can be used to obtain an $L^2$-Fourier restriction theorem for a general class of finite Borel measures satisfying
\begin{equation}\label{fourier}
|\hat{\mu}(\xi)|^2\lesssim |\xi|^{-\beta}, \forall \xi\in\mathbb R^d
\end{equation}
\begin{equation}\label{ahlfors}
\mu(B(x,r))\lesssim r^\alpha, \forall x\in\mathbb R^d, r>0
\end{equation}
where $0<\alpha, \beta<d$; they showed that \eqref{restriction} holds for $1\le p< p_0=\frac{4(d-\alpha)+2\beta}{4(d-\alpha)+\beta}$. Bak and Seeger \cite{bak-seeger} proved the same result for the endpoint $p_0$ and further strengthened it by replacing the $L^{p_0}$-norm with the $L^{p_0,2}$-Lorentz norm.

It is well known that if $\mu$ is the surface measure on a compact $C^\infty$ manifold then the sharpness can be tested by some version of Knapp's homogeneity argument. See e.g. the work by Iosevich and Lu \cite{iosevich-lu} who proved that if $\mu$ is the surface measure on a compact hypersurface and if $\mathcal F: L^{p_0}\to L^2(\mu)$, $p_0=\frac{2(d+1)}{d+3}$, then the Fourier decay assumption \eqref{fourier} is satisfied with $\alpha=d-1$. For general measures satisfying \eqref{fourier} and \eqref{ahlfors}, there is no Knapp's argument available to prove the sharpness of $p_0$. Here we show that indeed for certain measures the restriction estimate \eqref{restriction} holds in a range of $p$ beyond the range given above. This will follow from a restriction estimate based on an assumption on the $n$-fold convolution $\mu^{*n}=\mu*\dots*\mu$.

\begin{theorem}\label{theorem}
Let $\mu$ be a Borel probability measure on $\mathbb R^d$, let $1\le r\le \infty$ and assume that $\mu^{*n}\in L^r(\mathbb R^d)$.
Let $1\le p\le \frac{2n}{2n-1}$, if $r\ge 2$ and $1\le p\le \frac{nr'}{nr'-1}$, if $1\le r\le 2$, and let $1\le q\le\frac{p'}{nr'}$. Then
\begin{equation}\label{q-restriction}
\|\hat f\|_{L^q(\mu)} \lesssim  \|f\|_{L^p(\mathbb R^d)}, \forall f\in \mathcal S(\mathbb R^d).
\end{equation}
\end{theorem}

Apply Theorem \ref{theorem} with $n=2$, $r=\infty$, one obtains the following.

\begin{corollary}\label{corollary}
Let $\mu$ be a Borel probability measure on $\mathbb R^1$ such that $\mu*\mu\in L^\infty(\mathbb R^1)$. Then \eqref{restriction} holds for $1\le p\le 4/3$.
\end{corollary}

\noindent{\it Remarks.}
(i) It is not easy to construct measures supported on lower dimensional sets for which Corollary \ref{corollary} applies. Remarkably, K\"{o}rner showed by a combination of Baire category and probabilistic argument that there exist ``many'' Borel probability measures $\mu$ supported on compact sets of Hausdorff dimension 1/2 so that $\mu*\mu\in C_c(\mathbb R^1)$.

(ii) In Corollary \ref{corollary}, since $\mu*\mu$ satisfies \eqref{ahlfors} with $\alpha=1$, $\mu$ satisfies \eqref{ahlfors} with $\alpha=1/2$ (see Proposition \ref{convolution-ahlfors}). Suppose $\mu$ is supported on a compact set of Hausdorff dimension $\gamma$. It follows that $\gamma\ge 1/2$ (cf. \cite{wolff}, Proposition 8.2). Furthermore, if $\gamma< 1$, then $\beta$ and $\alpha$ in \eqref{fourier} and \eqref{ahlfors} can not exceed $\gamma$ (cf. \cite{wolff}, Corollary 8.7).

(iii) Under the above situation, since $\alpha,\beta\le\gamma$, the range of $p$ in \eqref{restriction} obtained from \cite{mockenhaupt}, \cite{mitsis}, \cite{bak-seeger} is no larger than $1\le p\le\frac{6-4\epsilon}{5-6\epsilon}$ where $\epsilon=\gamma-1/2$, while Corollary \ref{corollary} gives the range $1\le p\le 4/3$, which is an improvement if $\gamma<2/3$. However, we do not know any example of such a measure $\mu$ with $\beta$ (and $\alpha$) close to $\gamma$.

(iv) Suppose $\mu$ is as in Corollary \ref{corollary} and supported on a compact set of Hausdorff dimension 1/2. By Theorem \ref{theorem}, the restriction estimate \eqref{q-restriction} holds for $1\le p\le 4/3$, $1\le q\le p'/2$. By dimensionality considerations (see Proposition \ref{average-fourier} and Proposition \ref{knapp}), these are all the possible exponents $1\le p,q\le\infty$ for which \eqref{q-restriction} holds.
\section*{Proof of theorem \ref{theorem}}
The proof proceeds in a similar spirit as in \cite{rudin}, \cite{fefferman}. Fix a nonnegative function $\phi\in C^\infty_c(\mathbb R^d)$ that satisfies $\int_{\mathbb R^d} \phi(\xi)d\xi=1$. Let $\phi_\epsilon(\xi)=\epsilon^{-d}\phi(\xi/\epsilon)$ and $\mu_\epsilon(\xi)=\phi_\epsilon*\mu(\xi)=\int_{\mathbb R^d} \phi_\epsilon(\xi-\eta)d\mu(\eta)$. Since $\mu_\epsilon$ converges weakly to $\mu$, we have
$$\lim_{\epsilon\rightarrow 0}\int_{\mathbb R^d} |\hat f(\xi)|^q\mu_\epsilon(\xi) d\xi= \int_{\mathbb R^d} |\hat f(\xi)|^q d\mu(\xi)$$
for all $f\in\mathcal S(\mathbb R^d)$. Thus it suffices to show
$$\|\hat f\|_{L^q(\mu_\epsilon)}\le C\|f\|_{L^p(\mathbb R^d)}$$
where $C$ is a constant independent of $f$ and $\epsilon$. By H\"{o}lder's inequality, we may assume $q=\frac{p'}{nr'}$. Set $s=p'/n$. Note that by our assumption on the range of $p$, $s\ge2, q\ge 1$. By duality, we need to prove
\begin{equation}\label{dual}
\Big(\int_{\mathbb R^d} |\widehat {g\mu_\epsilon}(x)|^{ns}dx\Big)^{1/ns}\le C \Big(\int_{\mathbb R^d}|g(\xi)|^{q'}\mu_\epsilon(\xi)d\xi\Big)^{1/q'}
\end{equation}
for all bounded Borel function $g$. By the Hausdorff-Young inequality,
\begin{align*}
\Big(\int_{\mathbb R^d} |\widehat {g\mu_\epsilon}(x)|^{ns}dx\Big)^{1/s}
&=\Big(\int_{\mathbb R^d} |\widehat {g\mu_\epsilon}^n(x)|^s dx\Big)^{1/s}\\
&\le\Big(\int_{\mathbb R^d} |g\mu_\epsilon*\cdots*g\mu_\epsilon(\xi)|^{s'} d\xi\Big)^{1/s'}\\
&=\Big(\int_{\mathbb R^d} |\int_{\mathbb R^{(n-1)d}}G(\xi,\eta)M_\epsilon(\xi,\eta)d\eta|^{s'}d\xi\Big)^{1/s'}
\end{align*}
where $\eta=(\eta_1,\dots,\eta_{d-1})$, $\eta_0\equiv\xi$,
\begin{align*}
G(\xi,\eta)&= g(\eta_{n-1})\prod_{j=1}^{n-1} g(\eta_{j-1}-\eta_j),\\
M_\epsilon(\xi,\eta)&= \mu_\epsilon(\eta_{n-1})\prod_{j=1}^{n-1} \mu_\epsilon(\eta_{j-1}-\eta_j).
\end{align*}
Now by H\"{o}lder's inequality for the inner integral,
\begin{align*}
&\ \ \ \ \Big(\int_{\mathbb R^d} |\int_{\mathbb R^{(n-1)d}}G(\xi,\eta)M_\epsilon(\xi,\eta)d\eta|^{s'}d\xi\Big)^{1/s'}\\
&\le \left(\int_{\mathbb R^d} \Big(\mu_\epsilon^{*n}(\xi)\Big)^{s'/q} \Big(\int_{\mathbb R^{(n-1)d}}|G(\xi,\eta)|^{q'}M_\epsilon(\xi,\eta)d\eta\Big)^{s'/q'} d\xi\right)^{1/s'}
\end{align*}
Apply H\"{o}lder's inequality again, this is bounded by
\begin{align*}
&\ \ \ \ \|\mu_\epsilon^{*n}\|^{1/q}_r \Big(\int_{\mathbb R^d}\int_{\mathbb R^{(n-1)d}}|G(\xi,\eta)|^{q'}M_\epsilon(\xi,\eta)d\eta d\xi\Big)^{\frac{1}{s'}-\frac{1}{qr}}\\
&= \|\mu_\epsilon^{*n}\|^{1/q}_r \Big(\int_{\mathbb R^d}|g(\xi)|^{q'}\mu_\epsilon(\xi)d\xi\Big)^{n(\frac{1}{s'}-\frac{1}{qr})}\\
&\le \|\mu^{*n}\|^{1/q}_r \Big(\int_{\mathbb R^d}|g(\xi)|^{q'}\mu_\epsilon(\xi)d\xi\Big)^{n(\frac{1}{s'}-\frac{1}{qr})}
\end{align*}
where we have used Young's inequality in the last line. Since $\frac{1}{s'}-\frac{1}{qr}=\frac{1}{q'}$, we obtain \eqref{dual} after taking the $n$th root. \qed
\section*{Appendix}
For the sake of completeness, we include the proofs of the claims made in the remarks. Similar results can be found in \cite{mockenhaupt} and \cite {mitsis}.

\begin{proposition}\label{convolution-ahlfors}
Let $\mu$ be a Borel probability measure on $\mathbb{R}^d$. If $\mu^{*n}$ satisfies \eqref{ahlfors} with $0\le \alpha\le d$, then $\mu$ satisfies \eqref{ahlfors} with exponent $\alpha/n$.
\end{proposition}

\begin{proof}
Assume to the contrary that given $k$, $\mu(B_{r_k})\ge k r_k^{\alpha/n}$ for some ball $B_{r_k}$ with radius $r_k>0$. Let $B^*_{nr_k}=B_{r_k}+\cdots+B_{r_k}$ be the $n$-fold Minkowski sum, then
$$\mu^{*n}(B^*_{nr_k})\ge \mu(B_{r_k})^n\ge k^nr_k^{\alpha}.$$
On the other hand, since $\mu^{*n}$ satisfies \eqref{ahlfors},
$$\mu^{*n}(B^*_{nr_k})\lesssim (nr_k)^{\alpha}\lesssim r_k^{\alpha}, \forall k.$$
Let $k\rightarrow\infty$, we obtain a contradiction.
\end{proof}

\begin{proposition}\label{average-fourier}
Let $\mu$ be a Borel probability measure on $\mathbb{R}^d$ supported on a compact set of Hausdorff dimension $0\le\gamma<d$, then
$$\|\hat{\mu}\|_{s}=\infty, \forall 0<s<\frac{2d}{\gamma}.$$
\end{proposition}

\begin{proof}
Assume to the contrary that $\|\hat{\mu}\|_{s}<\infty$ for some $2<s<2d/\gamma$. Then
$$\dashint_{B(0,R)}|\hat{\mu}(\xi)|^2d\xi\le\Big(\dashint_{B(0,R)}|\hat{\mu}(\xi)|^{s} d\xi\Big)^{2/s}\lesssim R^{-2d/s}.$$
This decay in $R\rightarrow\infty$ implies $\gamma\ge 2d/s$ (cf. \cite{wolff}, Corollary 8.7). Since $2d/s>\gamma$, we obtain a contradiction.
\end{proof}

For the endpoint $s=\frac{2d}{\gamma}$ we have

\begin{proposition}\label{reflection}
Let $\mu$ be a Borel probability measure on $\mathbb R^d$ supported on a compact set $K$. Suppose $d/2\le \gamma<d$ and there exists $C\ge 1$ so that
\begin{equation}\label{AD}
C^{-1} r^\gamma\le\mu(B(x,r))\le C r^\gamma
\end{equation}
for all $x\in K$ and $0<r<1$. Then $\|\hat \mu\|_{\frac{2d}{\gamma}}=\infty$.
\end{proposition}

\begin{proof}
Assume to the contrary that $\|\hat \mu\|_{2d/\gamma}<\infty$. Let $\tilde\mu$ be the reflection of $\mu$, i.e. $\tilde\mu(A)=\mu(-A)$ for Borel sets $A$. Then $\widehat{\mu*\tilde\mu}=|\hat\mu|^2\in L^{d/\gamma}$. By the Hausdorff-Young inequality, this implies $\mu*\tilde\mu\in L^{(d/\gamma)'}$, and hence
\begin{equation}\label{ballholder}
\mu*\tilde\mu (B(0,\epsilon))\lesssim \|\mu*\tilde\mu\|_{L^{(d/\gamma)'}(B(0,\epsilon))}\epsilon^\gamma.
\end{equation}
On the other hand, by the upper regularity assumption in \eqref{AD} we can find $N_\epsilon$ many disjoint balls $B_j$ of radius $\epsilon/2$ centered in $K$ with $N_\epsilon\gtrsim \epsilon^{-\gamma}$. Since the difference set $B_j-B_j\subset B(0,\epsilon)$, we have
$$\mu*\tilde\mu(B(0,\epsilon))\gtrsim \sum^{N_\epsilon}_{j=1} \mu(B_j)^2\gtrsim N_\epsilon \epsilon^{2\gamma}\gtrsim\epsilon^{\gamma}$$
where we have used the lower regularity assumption in \eqref{AD} in the second inequality. Compare this with \eqref{ballholder} and notice that $\|\mu*\tilde\mu\|_{L^{(d/\gamma)'}(B(0,\epsilon))}\rightarrow 0$ as $\epsilon\rightarrow 0$, we obtain a contradiction.
\end{proof}

\begin{proposition}\label{knapp}
Let $\mu$ be a Borel probability measure on $\mathbb{R}^d$  supported on a compact set of Hausdorff dimension $0<\gamma\le d$. If \eqref{q-restriction} holds with $1\le p,q\le\infty$, then $q\le\frac{\gamma}{d}p'$.
\end{proposition}

\begin{proof}
Given $\epsilon>0$, by Billingsley's lemma (cf. \cite{falconer}, Proposition 4.9), there exist $x_0\in\mathbb{R}^d$ and $r_k\rightarrow 0$ such that $\mu(B(x_0,r_k))\gtrsim r^{\gamma+\epsilon}_k, \forall k$. For our purpose, we may assume $x_0=0$. Pick a bump function $\phi$ at $0$ and let $\hat f=\phi(\cdot/r_k)$ in \eqref{q-restriction}, we obtain $r^{(\gamma+\epsilon)/q}_k\lesssim r^{d/p'}_k, \forall k$. Comparing the powers then gives the desired result.
\end{proof}
\section*{Additional remarks}
(i) After submission of this paper, Hambrook and \L aba posted a preprint \cite{hambrook-laba} in which they provide examples of Cantor-type measures for which the range obtained from \cite{bak-seeger} is sharp.\\
\indent(ii) If $\mu$ is as in Corollary \ref{corollary} with compact support, then by Proposition \ref{reflection} it can not have lower regularity as in \eqref{AD} of degree 1/2.\\
\indent(iii) As pointed out by the referee, Corollary \ref{corollary} also follows from
$$\|f\mu*g\mu\|_{L^p(\mathbb R^d)}\le \|\mu*\mu\|_\infty^{1/p'}\|f\|_{L^p(\mu)}\|g\|_{L^p(\mu)},$$
which can be obtained by interpolating the cases $p=1, p=\infty$. See also Bak and McMichael \cite{bak-mcmichael}, Iosevich and Roudenko \cite{iosevich-roudenko}.

\section*{Acknowledgement}
The author would like to thank Andreas Seeger for suggesting this problem and a simplification of the original proof of the theorem which used a generalized coarea formula.

\end{document}